\documentclass[12pt,psamsfonts]{amsart}
\usepackage[top=1in, bottom=1in, left=1in, right=1in]{geometry}

\usepackage{amssymb,amsfonts,amsmath}
\usepackage{enumerate}
\usepackage[shortlabels]{enumitem}
\usepackage{comment}
\usepackage{hyperref} 
\usepackage[noabbrev,capitalise,nameinlink]{cleveref}
\usepackage{url}
\usepackage{todonotes}

\newtheorem*{corollary*}{Corollary}
\newtheorem*{theorem*}{Theorem}

\newtheorem{theorem}{Theorem}[section]
\newtheorem{corollary}[theorem]{Corollary}
\newtheorem{proposition}[theorem]{Proposition}
\newtheorem{lemma}[theorem]{Lemma}

\newtheorem{definition}[theorem]{Definition}

\theoremstyle{remark}

\numberwithin{remark}{subsection}

\numberwithin{equation}{section}

\renewcommand{\epsilon}{\varepsilon}

\renewcommand{\leq}{\leqslant}
\renewcommand{\geq}{\geqslant}
\renewcommand{\le}{\leqslant}
\renewcommand{\ge}{\geqslant}

\newcommand{\cB}{\mathcal{B}}

 \newcommand{\ind}{\mathbf{1}}
 \newcommand{\intmult}{\mathop{\int\dots\int}}

\newcommand{\E}{\mathbb{E}}
\newcommand{\cG}{\mathcal{G}}
\newcommand{\cL}{\mathcal{L}}

\newcommand{\R}{\mathbb{R}}

\renewcommand{\P}{\mathbb{P}}

\begin{document}

\title{A model problem for multiplicative chaos in number theory}

\author{Kannan Soundararajan}
\address{
Department of Mathematics \\ 
Stanford University   \\ 
450 Serra Mall, Building 380 \\
Stanford, CA \\
USA \\
94305}
\email{ksound@stanford.edu}

\author{Asif Zaman}
\address{
Department of Mathematics\\
University of Toronto \\
40 St. George Street, Room 6290 \\
Toronto, ON \\
CANADA \\
M5S 2E4}
\email{zaman@math.toronto.edu}

\date{\today}

\begin{abstract}  Resolving a conjecture of Helson, Harper recently established that partial sums of random multiplicative functions  
typically exhibit more than square-root cancellation.  Harper's work gives an example of a problem in number theory that is closely linked to ideas in probability theory connected with multiplicative chaos; another such closely related problem is the Fyodorov--Hiary--Keating conjecture on the maximum size of the Riemann zeta function in intervals of bounded length on the critical line.   In this paper we consider a problem that might be thought of as a simplified function field version of Helson's conjecture.  We develop and simplify the ideas of Harper in this context, with the hope that the simplified proof would be of use to readers seeking a gentle entry-point to this fascinating area.  
\end{abstract}

\maketitle
%\tableofcontents

\vspace*{-15pt}
%
% SECTION 1
%
\section{Introduction}
\label{sec:Intro}

The aim of this article is to describe, in a simple setting, some recent work of Harper \cite{Harper-2020a} on the distribution of random multiplicative functions.  The study of random multiplicative functions has been very active in recent years, and turns out to be closely related to problems of ``multiplicative chaos" which have recently received attention in the probability literature.   On the number theory side, the study of mean values of random multiplicative functions is closely related to problems involving the size of the Riemann zeta function in short intervals of the critical line, a line of investigation originating in the conjectures of Fyodorov, Hiary and Keating \cite{FyodorovHiaryKeating-2012}.   Let us begin by quickly describing the model problem that we study here, and then giving its connections with the problem of random multiplicative functions.

Consider a sequence $(X(k))_{k \geq 1}$ of independent standard complex Gaussians; thus the real and imaginary parts of $X(k)$ are distributed like independent real Gaussian random variables with mean $0$ and variance $\tfrac 12$.   Define a sequence of random variables $(A(N))_{N \geq 0}$ by the formal power series identity
\begin{equation}	
\exp\Big(\sum_{k = 1}^{\infty} \frac{X(k)}{\sqrt{k}} z^k \Big) = \sum_{n =0}^{\infty} A(n) z^n. 
\label{1.1}
\end{equation}	
The random variables $A(n)$ are naturally determined by the independent random variables $X(k)$; for example, $A(0) = 1$, $A(1) = X(1)$, $A(2) = X(1)^2/2+X(2)/\sqrt{2}$, and so on.   With this notation, the main result that we wish to explain is the following. 

\begin{corollary} \label{cor1.1}  For all $N\ge 1$, 
\begin{equation} \label{1.2} 
\E[|A(N)|^2 ] =1. 
\end{equation} 
However there are positive constants $C_1$ and $C_2$ such that for $N\ge 2$ 
\begin{equation} \label{1.3} 
\frac{C_1}{(\log N)^{\frac 14}}\le \E [|A(N)|]  \le \frac{C_2}{(\log N)^{\frac 14}}. 
\end{equation} 
In particular, $\E[|A(N)|] \to 0$ as $N\to \infty$.  
\end{corollary}

As mentioned already, the result above is motivated by a breakthrough of Harper on the partial sums of random multiplicative functions.   A random Steinhaus multiplicative function $f: {\Bbb N} \to \{ |z| =1\}$ is obtained by picking independent random variables $f(p)$ (for prime numbers $p$) distributed uniformly on the unit circle, and extending this to all natural numbers by (complete) multiplicativity.  Thus if $n= p_1^{\alpha_1} \cdots p_k^{\alpha_k}$ then $f(n)$ is the random variable $f(p_1)^{\alpha_1} \cdots f(p_k)^{\alpha_k}$.   Given such a random multiplicative function, an important goal is to understand the partial sums $\sum_{n\le x} f(n)$.   Since $\E[f(m) \overline{f(n)}] = 1$ if $m=n$ and $0$ otherwise it follows that 
\begin{equation} \label{1.4} 
\E \Big[ \Big|\sum_{n\le x} f(n)\Big|^2\Big]  = \lfloor x \rfloor = x+O(1). 
\end{equation} 
Harper showed that even though the variance is about $x$, surprisingly the typical size of  $\sum_{n\le x} f(n)$ is smaller than $\sqrt{x}$: 
\begin{equation} 
\label{1.5} 
\E \Big[ \Big| \sum_{n\le x} f(n) \Big| \Big] \asymp \frac{\sqrt{x}}{(\log \log x)^{\frac 14}}.
\end{equation} 
Here the relation $A \asymp B$ means that $C_1 B \le A \le C_2 B$ for some absolute positive constants $C_1$ and $C_2$.  Harper's result established the conjecture of Helson \cite{Helson-2010} that partial sums of random multiplicative functions typically exhibit more than square-root cancellation; the truth of Helson's conjecture seemed far from clear at the time, and indeed earlier work of Harper, Nikeghbali and Radziwi{\l \l} \cite{HarperNikeghbaliRadziwil-l-2015} had suggested the opposite of Helson's conjecture.  

Identifying $N$ with $\log x$, we see a strong parallel between the variances given in \eqref{1.2} and \eqref{1.4}, and the estimates for the first moment in \eqref{1.3} and \eqref{1.5}.  For large $k$, the quantity $\sum_{e^{k} < p \le e^{k+1}} f(p)$ behaves like a complex Gaussian with mean $0$ and variance 
$\sum_{e^k <p \le e^{k+1}} 1$, which by the prime number theorem is on the scale of $e^{k}(e-1)/k$.   After normalization, the sum over primes behaves analogously to $X(k)/\sqrt{k}$.   Then the random variable $A(N)$ is analogous to $e^{-N/2} \sum_{n\le e^{N}} f(n)$.   The key relation of the generating functions \eqref{1.1}   is paralleled by the Euler product formula 
$$
\sum_{n=1}^{\infty} \frac{f(n)}{n^s} = \prod_{p} \Big( 1- \frac{f(p)}{p^s}\Big)^{-1} = \exp\Big( \sum_{p^k}  \frac{f(p^k)}{kp^{ks}}\Big). 
$$ 

The analogy between our model problem and Harper's work is perhaps clearer in the ``function field setting."  Consider the polynomial ring ${\Bbb F}_q[t]$ where $q$ is a prime power, and ${\Bbb F}_q$ is a finite field with $q$ elements.  Many problems in the integers have close parallels in this polynomial ring, and for example a study of multiplicative functions in this framework may be found in \cite{GranvilleHarperSoundararajan-2015}.   The role of positive integers is played by ${\mathcal M}$, the set of monic polynomials.  Let ${\mathcal M}_n$ denote the monic polynomials of degree $n$, so that $|{\mathcal M}_n| =q^n$, which mirrors integers of size about $x$.   The role of primes is played by ${\mathcal P}$, the set of irreducible monic polynomials.  Letting ${\mathcal P}_n$ denote the monic irreducibles of degree $n$, there is also a well-known analogue of the prime number theorem (indeed of the Riemann hypothesis): $|{\mathcal P}_n| = q^n/n + O(q^{n/2}/n)$.    We can model Steinhaus multiplicative functions in this setting by considering (for monic irreducibles $P$) independent random variables $f(P)$ uniformly distributed on the unit circle, and then extending these to ${\mathcal M}$ by complete multiplicativity:  if $F = P_1^{\alpha_1} \cdots P_{k}^{\alpha_k}$ then put $f(F) = f(P_1)^{\alpha_1} \cdots f(P_k)^{\alpha_k}$.  
 
 In the function field context, our goal is to understand $\sum_{F \in {\mathcal M}_N} f(F)$.  If we set $A(n) = q^{-n/2} \sum_{F \in {\mathcal M}_n} f(F)$, then 
 $$ 
\sum_{n=0}^{\infty} A(n) z^n = \sum_{F \in {\mathcal M}} f(F) (q^{-1/2} z)^{\text{deg}(F)} = 
\prod_{P} \Big( 1- f(P)(q^{-\frac 12} z)^{\text{deg}(P)} \Big)^{-1}. 
$$
Writing 
$$ 
X(k) = \frac{\sqrt{k}}{q^{k/2}} \sum_{\substack{ \text{deg}(P)|k \\ r=k/\text{deg}(P) }} \frac{f(P)^{r}}{r}, 
$$ 
we may see that the generating function above equals  
$$
  \exp\Big( \sum_{P} \sum_{r=1}^{\infty} \frac{1}{r} f(P)^r (q^{-\frac 12} z)^{\text{deg}(P^r)}\Big) 
  = \exp\Big( \sum_{k=1}^{\infty} \frac{X(k)}{\sqrt{k}} z^k\Big). 
 $$
  This relation mirrors \eqref{1.1}.  Moreover, note that if $q^k$ is large then $X(k)$ is a (normalized) sum of about $q^k/k$ independent random variables uniformly distributed on the unit circle, so that $X(k)$ is distributed very nearly like a standard complex Gaussian.   In this manner we see that our model problem corresponds approximately to the study of Steinhaus multiplicative functions in the ${\Bbb F}_q[t]$ setting, and corresponds exactly to the limiting case when $q\to \infty$.

  Perhaps one of the earliest occurrences of the model \eqref{1.1} is in  the work of Hughes, Keating, O'Connell \cite{HughesKeatingOConnell-2001} where it arises as a
  prototypical example of a  log-correlated field.  Our interest arose in trying to simplify and understand Harper's work, and while we lectured on these results earlier (see for example \cite{BrudernMatomakiVaughanWooley-2020}), we have been slow with writing up this version.  In the meantime, independent work of Najnudel, Paquette and Simm \cite[(1.14) and Lemma 7.5]{NajnudelPaquetteSimm-2020} motivated by random matrix theory studies more general versions of this model (which they term ``holomorphic multiplicative chaos"), establishing the upper bounds in \cref{cor1.1} (and \cref{thm:main} below).

  \medskip 
  
% \sound{finish off with references to other work}
 
We conclude our introduction with a brief discussion of related  work.   In addition to the Steinhaus model of random multiplicative functions mentioned above, another natural model is the Rademacher class of random multiplicative functions taking values $\pm 1$ (independently and with equal probability) on the primes, and extended multiplicatively to all square-free numbers.  Harper \cite{Harper-2020a} also established the analogue of \eqref{1.5} in this class.  Indeed as we shall discuss in the next section, Harper \cite{Harper-2020a} determined the order of magnitude of the $2q$-th moment of partial sums for all $0\le q \le 1$, with the key feature being that the low moments exhibit more cancellation than what would be obtained by using H{\" o}lder's inequality with the second moment.   The complementary range of high moments is studied by Harper in \cite{Harper-2019a}.   While the partial sums of random multiplicative functions are typically smaller than expected, there are variant problems where the behavior follows expected Gaussian laws.  For example, the sums of random multiplicative functions over suitable short intervals or suitable arithmetic progressions \cite{ChatterjeeSoundararajan-2012}, or when restricted to integers without too many prime factors  \cite{Harper-2013}, \cite{Hough-2011}, exhibit Gaussian behavior.  

Given a random Steinhaus multiplicative function $f$, one can ask whether almost surely one has $\sum_{n\le x} f(n) =O(\sqrt{x})$  for all $x$.  That is, here we are choosing the multiplicative function $f$ at random, and asking about the behavior of partial sums as $x$ varies, in contrast with our earlier discussion where $x$ is first fixed and the random multiplicative function varies.  This problem, which is an analogue of the law of the iterated logarithm, was raised by Hal{\' a}sz \cite{Halasz-1983}, and investigated further in \cite{Harper-2013a} and \cite{LauTenenbaumWu-2013}.  Recently 
Hal{\' a}sz's problem was answered in the negative by Harper \cite{Harper-2020b}, who established that almost surely there are arbitrarily large $x$ with $|\sum_{n\le x} f(n) |\ge \sqrt{x} (\log \log x)^{\frac 14-\epsilon}$ for any $\epsilon >0$.  The law of the iterated logarithm shows that sums of independent random variables (for example taking values $\pm 1$ with equal probability) attain values as large as $\sqrt{x \log \log x}$ occasionally, and Harper's result differs from this by about the same amount $(\log \log x)^{-\frac 14}$ that appears in \eqref{1.5}.   Harper's result suggests that in our model problem, almost surely there exist arbitrarily large $N$ with  $|A(N)| \ge (\log N)^{\frac 14-\epsilon}$.  It would be of interest to make this precise, and perhaps obtain a more accurate law of the iterated logarithm in this context.

 Another problem in number theory that is closely related to this circle of ideas concerns the distribution of the Riemann zeta function over typical  intervals of length $1$ on the critical line Re$(s)=1/2$.  One vague connection between these problems is that $\zeta(\tfrac 12 +it)$ may be thought of as $\sum n^{-\frac 12 -it}$ for suitable ranges of $n$, and if $t$ is chosen randomly, the function $n^{it}$ behaves in some ways like a random Steinhaus multiplicative function.  More precisely, a conjecture of Fyodorov, Hiary, and Keating (see  \cite{FyodorovHiaryKeating-2012}, \cite{FyodorovKeating-2014}) states that if $t$ is chosen uniformly from $[T,2T]$ then  
 \begin{equation} 
 \label{1.6}
\log \log T -\tfrac 34 \log \log \log T -g(T) \le  \max_{t\le u\le t+1} \log |\zeta(\tfrac 12+iu)| \le \log \log T - \tfrac 34 \log \log \log T + g(T),
 \end{equation} 
 holds with probability tending to $1$ as $g(T)$ tends to infinity with $T$.   The key feature of this conjecture is the secondary term $-\frac 34 \log \log \log T$, which is smaller than the answer $-\frac 14 \log \log \log T$ that may be suggested by a crude application of Selberg's classical theorem that $\log |\zeta(\tfrac 12+it)|$ is distributed like a normal random variable with mean $0$ and variance $\sim \tfrac 12 \log \log T$.  Another closely related conjecture states that 
 \begin{equation} 
 \label{1.7} 
 \frac 1T \int_{T}^{2T} \Big( \frac 1{\log T} \int_0^1 |\zeta(\tfrac 12+it+ ih)|^2 dh \Big)^{\frac 12} \asymp \frac{1}{(\log \log T)^{\frac 14}}. 
 \end{equation} 
 Since $\int_T^{2T} |\zeta(\tfrac 12 +it)|^2 dt \sim T\log T$, the Cauchy--Schwarz inequality shows that the above quantity is $\ll 1$, and the interest above is that it is still smaller, and by a factor very similar to that arising in \eqref{1.5}.  Indeed there is a very strong analogy between \eqref{1.7} and \cref{prop3.2,prop8.2} below.   There has been a lot of recent progress towards the conjectures in \eqref{1.6} and \eqref{1.7} and other related questions, see   \cite{ArguinBeliusBourgadeRadziwil-lSoundararajan-2019},  \cite{ArguinBeliusHarper-2017}, \cite{ArguinBourgadeRadziwi-2020}, \cite{ArguinOuimetRadziwi-2019}, \cite{Gerspach-2020},  \cite{Harper-2020b}, \cite{Najnudel-2018}.   Most notably, the upper bound portion of \eqref{1.7} has been established by Harper \cite{Harper-2019b}, who also established a  slightly weaker version of the upper bound in \eqref{1.6}.  An even more precise version of the upper bound in \eqref{1.6} has been established by Arguin, Bourgade and Radziwi{\l \l}  \cite{ArguinBourgadeRadziwi-2020}.   For a recent comprehensive survey on this topic see \cite{BaileyKeating-2021}.
 
 The multiplication table problem (which asks for the number of integers $n$ up to $N^2$ that may be written as $n=ab$ with 
$a, b\le N$) exhibits some features in common with these problems, although the link here is perhaps less clearly defined.  We content ourselves with referring the reader to \cite{Ford-2008}, and pointing out also the interesting analogous problem of counting permutations in $S_{2N}$ that leave some $N$-element set fixed (so that such permutations may be thought of as the product of two permutations on $N$ element sets) \cite{EberhardFordGreen-2016}. 

Lastly, we mention that there is an extensive literature in probability where related problems are studied under ``critical multiplicative chaos"; a few sample references are  \cite{Berestycki-2017},   \cite{ChhaibiNajnudel-2019}, \cite{DuplantierRhodesSheffieldVargas-2014}, \cite{RhodesVargas-2014}. 

\medskip 

\subsection*{Acknowledgments}   K.S. is partially supported through a grant from the National Science Foundation, and a Simons Investigator Grant from the Simons Foundation.  Part of this work was written while K.S. was a senior Fellow at the ETH Institute for Theoretical Studies, whom he thanks for their warm and generous hospitality. Part of this work was written while A.Z. was a postdoctoral fellow at Stanford University supported by an NSERC Postdoctoral Fellowship. A.Z. is grateful to both institutions for their financial support and to the university for providing excellent working conditions. A.Z. also thanks Valeriya Kovaleva for helpful comments and sharing reference \cite{NajnudelPaquetteSimm-2020} with him. 

 %
% SECTION 2
%
\section{Preliminaries} 

In this section we set up some convenient notation, and make preliminary observations for our analysis of $A(N)$.

  By a partition $\lambda$ we mean a non-increasing sequence of 
non-negative integers $\lambda_1 \ge \lambda_2 \ge \ldots$, with $\lambda_n =0$ from some point onwards.   We denote by $|\lambda|$ the sum of the parts $\lambda_1 + \lambda_2 + \ldots$, and for an integer $k\ge 1$ we denote by $m_k =m_k(\lambda)$ the number of parts in $\lambda$ that are equal to $k$.  Given a partition $\lambda$, define 
\begin{equation} 
\label{2.1} 
a(\lambda) = a(\lambda;X) = \prod_k \Big(\frac{X(k)}{\sqrt{k}}\Big)^{m_k} \frac{1}{m_k!},  
\end{equation} 
where, as in the introduction, $X(k)$ is a sequence of independent standard complex Gaussians. 
With this notation, we have 
$$ 
\exp\Big(\sum_{k = 1}^{\infty} \frac{X(k)}{\sqrt{k}} z^k \Big) = \sum_{\lambda} a(\lambda) z^{|\lambda|}, 
$$ 
so that 
\begin{equation} 
\label{2.2} 
A(N) = \sum_{|\lambda|= N} a(\lambda). 
\end{equation}

Recall that a standard complex Gaussian $Z$ satisfies 
$$ 
\E \Big[ Z^m \overline{Z}^n \Big] = 
\begin{cases} 
n! &\text{if } m=n, \\ 
0 &\text{if } m\neq n. 
\end{cases}
$$  
 It follows that if $\lambda$ and $\lambda'$ are two different partitions then 
 \begin{equation}
 \E [ a(\lambda) \overline{a(\lambda^\prime)} ] =0, 	
 \label{eqn:PartitionOrthogonality}
 \end{equation}
 while if $\lambda = \lambda'$ then 
 \begin{equation}
 \label{eqn:PartitionSquare}
 \E [ |a(\lambda)|^2 ]  = \prod_{k} \frac{1}{m_k!^2 k^{m_k}} \E [ |X(k)|^{2m_k} ] = \prod_{k} \frac{1}{m_k! k^{m_k}}, 
 \end{equation}
 where we again use the notation that the partition $\lambda$ contains $m_k$ parts equal to $k$.  
 We deduce that 
  \begin{equation*} 
%  \label{2.2} 
\E [|A(N)|^2] = \sum_{|\lambda|= N} \E [ |a(\lambda)|^2 ]= \sum_{|\lambda|= N} \prod_{k} \frac{1}{m_k! \cdot k^{m_k}} = 1, 
\end{equation*} 
where the last step follows from the familiar formula for the number of permutations in $S_N$ whose cycle decomposition corresponds to the partition $\lambda$.  
This establishes \eqref{1.2}, and our main task is to bound the first moment in \eqref{1.3}.  

In fact we will determine the order of magnitude of all low moments $\E[|A(N)|^{2q}]$ for $0\le q\le 1$, following Harper who determined the order of magnitude for such moments for random multiplicative functions. 
 
\begin{theorem} \label{thm:main}
	Uniformly for any integer $N \geq 1$ and any $0 \leq q \leq 1$, 
	\[
	\E \big[ |A(N)|^{2q} \big] \asymp \Big(\frac{1}{(1-q) \sqrt{\log N} +1}\Big)^q. 
	\]	
	 In particular \eqref{1.3} holds.  
\end{theorem}

Naturally one can study $A(N)$ through the generating function $\sum_{n=0}^{\infty} A(n)z^n$, which converges almost surely for $|z| <1$.  For example, by Cauchy's theorem we have for $r<1$ the almost sure identity 
$$ 
A(N) = \frac{1}{2\pi i} \int_{|z|=r} \sum_{n=0}^{\infty} A(n) z^n \frac{dz}{z^{N+1}}.   
$$ 
Indeed since $A(N)$ depends only on the random variables $X(k)$ for $k\le N$, we can avoid all issues of convergence and write (for any $K\ge N$ and any $r>0$) 
$$ 
A(N) = \frac{1}{2\pi i} \int_{|z|=r} \exp\Big( \sum_{k \leq K} \frac{X(k)}{\sqrt{k}} z^k \Big) \frac{dz}{z^{N+1}}. 
$$ 
We will not   use this relation, but it motivates us to define (for any real number $K\ge 1$) 
\begin{equation} 
\label{2.3} 
F_K(z) = \exp\Big( \sum_{k\le K} \frac{X(k)}{\sqrt{k}} z^k \Big). 
\end{equation}  
When the parameter $K$ is clear from context, we shall abbreviate $F_K(z)$ to $F(z)$.

We shall relate the problem of bounding $\E [|A(N)|^{2q}]$ to that of estimating 
\begin{equation} 
\label{2.4} 
\E \Big[ \Big( \frac{1}{2\pi} \int_0^{2\pi} |F_K(re^{i\theta})|^2 d\theta \Big)^{q} \Big]. 
\end{equation} 
Here $K$ will be a parameter of size about $N$, and $r$ will be a parameter close to $1$.  

\begin{lemma} \label{lem3}  For any $K\ge 1$, any $r>0$, and any $\theta \in \R$, we 
have 
$$ 
\E [ |F_K(re^{i\theta})|^2 ] = \exp\Big( \sum_{k\le K} \frac{r^{2k}}{k}\Big). 
$$ 
\end{lemma} 
\begin{proof}  Since the complex Gaussian is rotationally symmetric, the variables $X(k)$ and $X(k)e^{ik\theta}$ are identically distributed and therefore $|F_K(re^{i\theta})|^2$ is distributed identically as $|F_K(r)|^2$.  Now $\text{Re } \sum_{k\le K} X(k)r^k/\sqrt{k}$ is a sum of independent Gaussians, and therefore is distributed like a real Gaussian with mean $0$ and variance $\frac 12 \sum_{k\le K} r^{2k}/k$.  The lemma follows upon recalling that if $Z$ is a real Gaussian with mean zero and variance $\sigma^2$ then $\E [e^{tZ}] =  e^{t^2 \sigma^2/2}$. 
\end{proof} 

From \cref{lem3} and H{\" o}lder's inequality it follows that 
\begin{equation} 
\label{2.5} 
\E \Big[ \Big(\frac{1}{2\pi} \int_0^{2\pi} |F_K(re^{i\theta})|^2 d\theta\Big)^q \Big] \le \Big( \E \Big[ \frac{1}{2\pi} \int_0^{2\pi} |F_K(re^{i\theta})|^2 d\theta \Big] \Big)^q = \exp\Big( q\sum_{k\le K} \frac{r^{2k}}{k} \Big).  
\end{equation} 
Thinking of $r=1$ and $K=N$ for simplicity (this is the most relevant range of the parameters) this furnishes an upper bound of size $N^q$ above.  The true size of the quantity on the left side above turns out to be a little bit smaller, by a factor $(1+(1-q)\sqrt{\log N})^q$ exactly as in \cref{thm:main}.   We point out that there are close parallels between the moment in \eqref{2.5} (for $q=1/2$) and the corresponding problem for 
$|\zeta(\tfrac 12+it)|$ considered in \eqref{1.7}.

Ultimately the smaller size of the left side of \eqref{2.5} can be traced to the ``ballot problem" in probability theory.   We will borrow from Harper  the following extension of classical results on Gaussian random walks.

\begin{lemma}[Harper] \label{lem:Harper}
	Let $a \geq 1$. For any integer $n \geq 1$, let $G_1,\dots,G_n$ be independent real Gaussian random variables, each having mean zero and variance between $\frac{1}{20}$ and $20$, say. Let $h$ be a function such that $|h(j)| \leq 10 \log j$. Then 
	\[
	\P\Big( \sum_{m=1}^j G_m \leq a  + h(j), \forall 1 \leq j \leq n \Big) \asymp \min\Big( 1, \frac{a}{\sqrt{n}} \Big).
	\]
\end{lemma}
\begin{proof} This is \cite[Probability Result 1, p. 29]{Harper-2020a}. 
	 Harper states the result  with $a$ and $n$ being large, 
	but this may be relaxed by adjusting the implied constants suitably.  
\end{proof}

 	The term $h(j)$ in \cref{lem:Harper} is needed, for example, in the upper bound part of our argument to obtain convergence of some sums.  It should be thought of as largely harmless, since a sum of $j$ independent Gaussian variables would typically exhibit fluctuations on the scale of $\sqrt{j}$ and $h(j)$ is negligible in comparison to this natural scale.

With the one exception of \cref{lem:Harper} above, we have kept the proof of \cref{thm:main}   self-contained.  For convenience, we have split the paper into two parts, focussing first on the upper bound implicit in \cref{thm:main} (see  Sections \ref{sec:ProofUpper} to \ref{sec:BallotSavings}) and then dealing with the lower bound (see Sections \ref{sec:ProofLower} to \ref{sec:TiltedProbability}). Note that the relation $A \ll B$ means that $A \leq C B$ for some absolute positive constant $C$. \\

%/////////////////////////////////////////////
%	PART I - The upper bound
%/////////////////////////////////////////////

\begin{center}
{\bf Part I: The upper bound of \cref{thm:main}}
\end{center}

%
% SECTION 3
%
\section{Deducing the upper bound from two propositions}
\label{sec:ProofUpper}

 It is enough to prove the upper bound in the range $\tfrac 12 \le q \le 1$, since by H{\" o}lder's inequality 
 the bound would then hold for all smaller $q$ as well.  As mentioned earlier, the problem of bounding $\E[|A(N)|^{2q}]$ may be related to the problem of bounding moments of the generating function $F_K(re^{i\theta})$ as in \eqref{2.4}.  We make this link precise here, and reduce 
  the upper bound part of the main theorem to two propositions that will be established in the following sections. 

\begin{proposition} 
\label{prop3.1} For $1/2\le q \le 1$ and any integer $N \geq 1$, we have 
$$ 
\Big(\E[ |A(N)|^{2q} ]\Big)^{\frac 1{2q}} \ll \frac{1}{\sqrt{N}} \sum_{j=1}^{J} \Big( \E \Big[ \Big(\frac{1}{2\pi} \int_0^{2\pi}  |F_{N/2^j}(\exp(j/N+i\theta))|^2 d\theta\Big)^q\Big]\Big)^{\frac{1}{2q}} + \frac{1}{ N},
$$ 
where $J= \lceil 4\log \log(4 N) \rceil$.
\end{proposition} 

\begin{proposition} 
\label{prop3.2} 
Let $K\ge 1$ be a real number and let $F(z)= F_K(z)$ be as in \eqref{2.3}.   Uniformly for $1/2 \le q \le 1$ and $1\le r\le e^{1/K}$ we have 
$$ 
\E\Big[ \Big(\frac{1}{2\pi } \int_0^{2\pi} |F(re^{i\theta})|^2 d\theta \Big)^q \Big] \ll \Big( \frac{K}{(1-q)\sqrt{\log K} + 1} \Big)^q. 
$$
\end{proposition}

Applying the estimate of \cref{prop3.2} in \cref{prop3.1} it follows that for $1/2 \le q\le 1$ 
$$ 
\Big(\E[ |A(N)|^{2q} ]\Big)^{\frac 1{2q}} \ll \frac{1}{\sqrt{N}} \sum_{j=1}^{J} \Big( \frac{N/2^j}{(1-q) \sqrt{\log N} +1 } \Big)^{\frac 12} + \frac{1}{N} 
\ll \Big( \frac{1}{(1-q)\sqrt{\log N}+1}\Big)^{\frac 12}.
$$
This establishes the upper bound in \cref{thm:main} in the range $1/2\le q\le 1$.

%
% SECTION 4
%
\section{Proof of \cref{prop3.1}}
\label{sec:DecomposeUpper}

Our starting point is the representation of $A(N)$ as a sum over partitions $\lambda$ of $N$, recall \eqref{2.1} and \eqref{2.2}.  
Group these partitions according to the size of their largest part $\lambda_1$.  
For $1\le j \le J$ write 
$$ 
A_j(N) = \sum_{\substack{ |\lambda| = N \\ N/2^{j} <\lambda_1 \le N/2^{j-1}} }a(\lambda), 
$$
and put 
$$ 
\widetilde{A}_J(N) =  \sum_{\substack{|\lambda|= N \\ \lambda_1 \le N/2^J}} a(\lambda), 
$$
so that we have the natural decomposition
$$ 
A(N) = \sum_{j=1}^{J} A_j(N) + {\widetilde A}_J(n).  
$$ 
Minkowski's inequality gives, for $1/2 \le q\le 1$,  
\begin{equation} 
\label{4.1}
\Big( \E[ |A(N)|^{2q}]\Big)^{\frac{1}{2q}} \le \sum_{j=1}^{J} \Big( \E [ |A_j(N)|^{2q} ]\Big)^{\frac 1{2q}} + \Big(\E [ |\widetilde{A}_J(N)|^{2q}] \Big)^{\frac{1}{2q}} .
\end{equation}

We begin by estimating the last term in the right side of \eqref{4.1}, showing that it is $\ll 1/N$.  By H{\" o}lder's inequality, we find that 
\begin{equation}
\label{4.2} 
\Big( \E [ |\widetilde{A}_J(N)|^{2q}]\Big)^{\frac{1}{q}} \le \E [ |\widetilde{A}_J(N)|^2 ] =\sum_{\substack{ |\lambda|= N \\ \lambda_1 \le N/2^J}} 
\prod_{k} \frac{1}{m_k! k^{m_k}}. 
\end{equation}  
The right side of \eqref{4.2} is the proportion of elements in the symmetric group $S_N$ whose cycle decomposition has largest 
cycle $\le N/2^J$ in length, and it is a familiar fact that such permutations are rare (corresponding to the rarity of integers all of whose prime factors are small).    
We may supply a quick bound (corresponding to Rankin's trick with Dirichlet series) on this quantity  as follows.  The right side of \eqref{4.2} is the coefficient of $z^N$ in the generating function $\exp( \sum_{k\le N/2^J} z^k/k)$.  
Since the coefficients of this generating function are all non-negative, for any $r>0$ we conclude that the right side of \eqref{4.2} is 
$$
\le r^{-N} \exp\Big( \sum_{k\le N/2^J} \frac{r^k}{k} \Big) \ll \exp(-2^J) \frac{N}{2^J} \ll \frac{1}{N^2}, 
$$ 
upon choosing $r= \exp(2^J/N)$.  Inserting this into \eqref{4.2} we conclude that the last term in \eqref{4.1} is $\ll 1/N$.

We now focus on bounding the contribution of the $j$-th term of the sum in \eqref{4.1}.   Recall that the  term $A_j(N)$ sums over partitions 
$\lambda$ of $N$ with largest part $\lambda_1$ lying between $N/2^j$ and $N/2^{j-1}$.   Decompose the partition $\lambda$ into $\rho$ and $\sigma$, where $\rho$ consists of those non-zero parts in $\lambda$ that lie between $N/2^j$ and $N/2^{j-1}$, and $\sigma$ consists of those non-zero parts of $\lambda$ that are $\le N/2^j$; note that $\rho$ must have at least one non-zero term. It follows from \eqref{2.1} that $a(\lambda) = a(\rho)a(\sigma)$.   Thus, with the above understanding,  
$$ 
A_j(N) = \sum_{\substack{ \rho, \sigma \\ |\rho|+ |\sigma|=N \\ |\rho| >0}}  a(\rho) a(\sigma). 
$$ 
Observe that $a(\sigma)$ depends only on the random variables $X(k)$ for $k\le N/2^j$, while $a(\rho)$ depends only on the 
random variables $X(k)$ with $N/2^j < k \le N/2^{j-1}$.  

We shall bound the expected value of $|A_j(N)|$ by first conditioning on the variables $X(k)$ for $k\le N/2^j$ (so that $a(\sigma)$ is 
fixed in the notation above), and then bounding the expectation over these small variables $X(k)$.  Let ${\Bbb E}_j$ denote the 
conditional expectation when the variables $X(k)$ for $k\le N/2^j$ are fixed.  We shall show that 
\begin{equation} 
\label{4.3} 
\E_j \Big[ |A_j(N)|^{2q}\Big] \ll \Big( \frac{1}{2\pi N} \int_0^{2\pi} |F_{N/2^j} (\exp(j/N+i\theta))|^2 d\theta\Big)^q,
\end{equation} 
so that, upon now taking the expectation over the variables $X(k)$ with $k\le N/2^j$, we may conclude that
$$ 
\E \Big[ |A_j(N)|^{2q}\Big] \ll \E \Big[ \Big( \frac{1}{2\pi N} \int_0^{2\pi} |F_{N/2^j} (\exp(j/N+i\theta))|^2 d\theta\Big)^q\Big].
$$
This would complete the proof of our proposition. 

It remains now to establish \eqref{4.3}.  By H{\" o}lder's inequality, 
$$ 
\E_j \Big[ |A_j(N)|^{2q}\Big] \le \E_j\Big[ |A_j(N)|^2 \Big]^q,
$$ 
so that \eqref{4.3} follows from the estimate 
\begin{equation} 
\label{4.4} 
\E_j\Big[ |A_j(N)|^2 \Big] \ll \Big( \frac{1}{2\pi N} \int_0^{2\pi} |F_{N/2^j} (\exp(j/N+i\theta))|^2 d\theta\Big). 
\end{equation}

Expanding out the expression for $A_j(N)$ in terms of $a(\rho)$ and $a(\sigma)$, we find 
$$ 
\E_j \Big[ |A_j(N)|^2 \Big] = \sum_{\substack{ \rho_1, \sigma_1 \\ |\rho_1|+|\sigma_1| =N \\ |\rho_1| > 0}} 
\sum_{\substack{\rho_2, \sigma_2 \\ |\rho_2|  + |\sigma_2| =N \\ |\rho_2| >0}} a(\sigma_1) \overline{a(\sigma_2)}  \E_j \Big[ a(\rho_1) \overline{a(\rho_2)}\Big].
$$ 
Now  note  \eqref{eqn:PartitionOrthogonality} implies that $\E_j[ a(\rho_1) \overline{a(\rho_2)}] = 0$ unless $\rho_1 = \rho_2$.  Thus, writing $n= |\rho_1| = |\rho_2|$ 
with $N/2^j < n \le N$ (and so $|\sigma_1| = |\sigma_2| = N-n$), we obtain 
\begin{equation} 
\label{4.5}
\E_j \Big[ |A_j(N)|^2 \Big] = \sum_{N/2^j < n \le N} \Big| \sum_{ |\sigma| =N-n} a(\sigma)\Big|^2 \sum_{|\rho|=n} \E_j \Big[ |a(\rho)|^2\Big].
\end{equation} 
 
 To proceed further, we must estimate the sum over the partitions $\rho$ above.  Recall that all parts of $\rho$ must 
 be between $N/2^j$ and $N/2^{j-1}$, and denote (as before) by $m_k$ the number of parts of size $k$ in $\rho$ (so $N/2^j < k \le N/2^{j-1}$). 
 Then by \eqref{eqn:PartitionSquare}, we have that
 $$
\E_j \Big[ |a(\rho)|^2 \Big] = \prod_{N/2^j < k \le N/2^{j-1}} \frac{1}{m_k! k^{m_k}} \le   \Big( \frac{2^j}{N}\Big)^r,
$$ 
where $r$ denotes the number of parts in $\rho$ (so that $2^{j-1}n/N \le r < 2^j n/N$).  Thus
$$
 \sum_{|\rho|=n} \E_j \Big[ |a(\rho)|^2\Big] \le \sum_{2^{j-1}n/N \le r < 2^j n/N} \#\{ \rho: \ \ |\rho|=n, \rho \text{ has } r \text{ parts} \} \Big(\frac{2^j}{N}\Big)^r.
 $$
 The number of partitions $\rho$ of $n$ with $r$ parts (all between $N/2^j$ and $N/2^{j-1}$) 
 is at most $\binom{[N/2^{j}] +r}{r-1}$, since $r-1$ parts may be freely chosen among the integers in $(N/2^j, N/2^{j-1}]$ and the final part 
 has at most one possibility.  Thus, using $r\le 2^j \le N/2^j$ for large $N$, 
 \begin{align*}
\sum_{|\rho|=n} \E_j\Big[ |a(\rho)|^2\Big] &\le \sum_{2^{j-1}n/N \le r <2^{j}n/N} \Big(\frac{2^j}{N}\Big)^r \frac{(N/2^j+r)^{r-1}}{(r-1)!} \le \frac{2^j}{N} \sum_{2^{j-1}n/N \leq r < 2^jn/N} \frac{2^{r-1}}{(r-1)!}. 
\end{align*}
If $n\le N/2$ then the sum over $r$ above is $\ll 1$, while if $n> N/2$ we may bound the sum over $r$ above by $\ll 2^{-j}$.   Thus in both cases we may conclude that 
\begin{align*}
\sum_{|\rho|=n} \E_j\Big[ |a(\rho)|^2\Big] &\ll \frac 1N \exp\Big( 2j \frac{N-n}{N}\Big). 
\end{align*} 

Inserting the above in \eqref{4.5} it follows that 
$$ 
\E_j \Big[ |A_j(N)|^2 \Big] \ll   \frac1N  \sum_{N/2^j < n\le N} \Big| \sum_{|\sigma|= N-n} a(\sigma)\Big|^2 \exp\Big( 2j\frac{N-n}{N}\Big).
$$
Now 
$$ 
F_{N/2^j}(z) = \sum_{r} \Big( \sum_{|\sigma|=r} a(\sigma)\Big) z^r, 
$$ 
and so by Parseval 
$$
 \frac1N  \sum_{N/2^j < n\le N} \Big| \sum_{|\sigma|= N-n} a(\sigma)\Big|^2 \exp\Big( 2j\frac{N-n}{N}\Big) 
 \le  \frac{1}{2\pi N} \int_0^{2\pi} |F_{N/2^j}(\exp(j/N+i\theta))|^2 d\theta.
$$ 
This establishes \eqref{4.4} and completes the proof of the proposition.  
     \hfill \qed

%
% SECTION 5
%
\section{Plan for the proof of \cref{prop3.2}}

The  proof  of the upper bound portion of  \cref{thm:main} has  now been reduced to establishing \cref{prop3.2}.  By \cref{lem3} we conclude that for $1 \leq r \leq e^{1/K}$,
\begin{equation} 
\label{5.1} 
\E \Big[ \frac{1}{2\pi} \int_0^{2\pi } |F(re^{i\theta})|^2 d\theta \Big] = \exp\Big( \sum_{k\le K} \frac{r^{2k}}{k} \Big) \asymp K.
\end{equation} 
By H{\" o}lder's inequality it follows that for $0< q\le 1$ and $1 \leq r \le e^{1/K}$, 
 \begin{equation*} 
% \label{5.2} 
 \E \Big[ \Big(\frac{1}{2\pi} \int_0^{2\pi } |F(re^{i\theta})|^2 d\theta \Big)^q \Big]  \ll K^q,  
 \end{equation*} 
 so that in \cref{prop3.2} we are looking for a small improvement over this easy upper bound.   
 As mentioned earlier, the source of this improvement is connected to the ballot problem in probability.

 \begin{definition}  \label{def5.1} Let $K \geq 3$. Suppose $1 \leq r\le e^{1/K}$ and  $1\le A \le \sqrt{\log K}$.  Define 
 $\cG_r(A,\theta;K)$  to be the following event:  for each $1\le n \le \log K$, one has
\[
 \sum_{k< e^n} \Big( \mathrm{Re} \frac{X(k) r^k e^{ik\theta}}{\sqrt{k}}  - \frac{r^{2k}}{k} \Big) \le A+ 10 \log n. 
\]
 Further, define $\cG_r(A;K)$ to be the event where $\cG_r(A,\theta;K)$ holds for all $\theta \in [0,2\pi)$.   
 \end{definition}

We shall deduce \cref{prop3.2} from the following two propositions concerning $\cG_r(A;K)$.

\begin{proposition} \label{prop5.2}  For all $1 \leq r \leq e^{1/K}$ and all $1\le A \le \sqrt{\log K}$ we have 
$$ 
\P[ \cG_r(A;K) \, \mathrm{ fails}] \ll \exp(-A). 
$$ 
\end{proposition}

 Thus the event $\cG_r(A;K)$ is very likely for large $A$.  The crucial point is that on this 
 large set, we can improve upon the upper bound \eqref{5.1}  for the mean square of $F(re^{i\theta})$.
 
 \begin{proposition} \label{prop5.3}  With notations as above,  we have 
 $$ 
 \E \Big[ \ind_{\cG_r(A,\theta;K)} |F(re^{i\theta})|^2 \Big] \asymp \frac{A}{\sqrt{\log K}} K, 
 $$ 
 and therefore 
 $$ 
 \E \Big[ \ind_{\cG_r(A;K)} \int_0^{2\pi} |F(re^{i\theta})|^2 d\theta \Big] \ll \frac{A}{\sqrt{\log K}} K. 
 $$ 
 \end{proposition}

 \begin{proof}[Deducing \cref{prop3.2} from \cref{prop5.2,prop5.3}] Assume without loss that $K \geq e^4$. Partition the 
 whole probability space into the events $\cG_r(1;K)$, $\cG_r(2^j;K)\backslash \cG_r(2^{j-1};K)$ for $1\le j\le J:= \lfloor (\log \log K)/(2\log 2) \rfloor$, and 
 $\cG_r(2^J;K)^c$.  \cref{prop5.3} and H{\" o}lder's inequality give 
$$ 
\E \Big[ \ind_{\cG_r(1;K)} \Big(\int_0^{2\pi} |F(re^{i\theta})|^2 d\theta \Big)^q \Big] \le \Big( \E \Big[ \ind_{\cG_r(1;K)}\int_0^{2\pi} |F(re^{i\theta})|^2 d\theta 
\Big] \Big)^q \ll \Big( \frac{K}{\sqrt{\log K}}\Big)^q. 
$$ 
For $1\le j\le J$, the event $\cG_r(2^j;K)\backslash \cG_r(2^{j-1};K)$ means that $\cG_r(2^j;K)$ holds but $\cG_r(2^{j-1};K)$ fails.  Thus 
H{\" o}lder's inequality gives 
\begin{align*} 
\E \Big[ \ind_{\cG_r(2^j;K)\backslash\cG_r(2^{j-1};K)} &\Big(\int_0^{2\pi} |F(re^{i\theta})|^2 d\theta \Big)^q \Big]  \\
&\le \Big( \P\Big[ \cG_r(2^{j-1};K) \text{ fails}\Big] \Big)^{1-q} \Big( \E \Big[ \ind_{\cG_r(2^j;K)} \int_0^{2\pi} |F(re^{i\theta})|^2 d\theta \Big]\Big)^q. 
\end{align*} 
 Using \cref{prop5.2,prop5.3}, we see that this is 
 $$ 
 \ll \exp( - (1-q)2^{j-1}) \Big( \frac{2^j K}{\sqrt{\log K}}\Big)^q \le \Big( \frac{K}{\sqrt{\log K}}\Big)^q \Big(2^j \exp(-(1-q)2^{j-1})\Big). 
 $$
 Similarly we find that 
 $$ 
 \E \Big[ \ind_{\cG_r(2^{J};K)^c} \Big(\int_0^{2\pi} |F(re^{i\theta})|^2 d\theta \Big)^q \Big] \ll K^q \exp\Big( -(1-q) 2^J\Big).
 $$ 
 Adding up these estimates, we deduce \cref{prop3.2}. 
 \end{proof}

%
% SECTION 6
%
\section{Proof of \cref{prop5.2}}  

\noindent We prove that the event $\cG_r(A;K)$ is very likely by showing that $\sum_{k< e^n} \text{Re}  X(k)r^ke^{ik\theta}/ \sqrt{k}$ is unlikely to exceed the stated barrier for any single $1 \leq n \leq \log K$ and any single $\theta \in [0,2\pi]$. 
From \cref{def5.1}, the union bound gives
\begin{align} 
\label{6.1} 
\P[ \cG_r(A;K) \text{ fails}] 
&\le \sum_{n\le \log K} \P\Big[ \max_{\theta} \sum_{k< e^n} \text{Re} \frac{X(k)r^ke^{ik\theta}}{\sqrt{k}} \ge A+10\log n +\sum_{k< e^n} 
\frac{r^{2k}}{k} \Big] \nonumber\\
&=: \sum_{n\le \log K} {\mathcal P}_n. 
\end{align}
 
 To estimate ${\mathcal P}_n$ we discretize the possible values of $\theta \in [0,2\pi)$, setting $\theta_j = 2\pi j/\lceil ne^n\rceil$ 
 for $0\le j < \lceil ne^n\rceil$.   If the maximum over $\theta$ in the definition of $\mathcal P_n$ satisfies the corresponding inequality, then we must have either 
 \begin{equation} 
 \label{6.2} 
 \sum_{k< e^n} \text{Re} \frac{X(k)r^ke^{ik\theta_j}}{\sqrt{k}} \ge  \frac A2 +5\log n +\sum_{k< e^n} 
\frac{r^{2k}}{k}, \qquad \text{ for some } 0\le j < \lceil ne^n\rceil, 
\end{equation} 
or for some $0 \leq j < \lceil n e^n \rceil$, and some $\theta_j < \theta\le \theta_{j+1}$ we must have 
$$ 
\text{Re} \int_{\theta_j}^\theta \sum_{k< e^n} X(k) r^k (i\sqrt{k} e^{iky}) dy =  \text{Re} \sum_{k< e^n} \frac{X(k)r^k}{\sqrt{k}} (e^{i\theta k}-e^{i\theta_j k})  
\ge \frac A2 + 5\log n. 
$$ 
This second case only happens if 
\begin{equation} 
\label{6.3} 
\int_{\theta_j}^{\theta_{j+1}} \Big| \sum_{k< e^n} X(k) r^k \sqrt{k} e^{iky} \Big| dy \ge \frac{A}{2} + 5\log n, \qquad \text{ for some } 0\le j < \lceil ne^n\rceil. 
\end{equation} 
Since the random variables $X(k)$ are independent and rotationally invariant, it follows that 
\begin{equation} 
\label{6.4} 
\mathcal P_n \le ne^n \Big( \mathcal P_n^{'} + \mathcal P_n^{''} \Big), 
\end{equation} 
 where ${\mathcal P}_n^{'}$ is the probability that the inequality in \eqref{6.2} holds for $j=0$, and ${\mathcal P}_n^{''}$ is the 
 probability that the inequality in \eqref{6.3} holds for $j=0$.  
 
 Since $\text{Re} \sum_{k< e^n} X(k)r^k/\sqrt{k}$ is a real Gaussian with mean $0$ and variance $\tfrac 12 \sum_{k< e^n} r^{2k}/k$, 
 it follows that 	
 \begin{equation*} 
 \mathcal P_n^{'} \ll \exp\Big( - \Big(\sum_{k< e^n} \frac{r^{2k}}{k} + \frac A2 + 5\log n \Big)^2{\Big /}\sum_{k< e^n} \frac{r^{2k}}{k} \Big) 
 \le \exp\Big( - \sum_{k< e^{n} } \frac{r^{2k}}{k} - A -10 \log n\Big), 
 \end{equation*} 
 where we used that $\int_{t}^{\infty} e^{-x^2/2} dx \ll e^{-t^2/2}$ for $t\ge 0$.   Thus, as $1 \leq r \leq e^{1/K}$, 
 \begin{equation} 
 \label{6.5} 
 {\mathcal P}_n^{'} \ll \frac{e^{-n-A}}{n^{10}}. 
 \end{equation} 
 
 Now we turn to the task of estimating ${\mathcal P}_n^{''}$.   By H{\" o}lder's inequality, it follows 
 that if \eqref{6.3} holds (with $j=0$ there) then for any integer $\ell \ge 1$ we must have 
 $$ 
 \theta_1^{2\ell-1} \int_{0}^{\theta_1} \Big| \sum_{k< e^n} X(k) r^k \sqrt{k} e^{iky} \Big|^{2\ell}  dy \ge \Big( \frac A2 + 5\log n \Big)^{2\ell}.
 $$ 
 Therefore, by Chebyshev's inequality, 
 \begin{align*}
 {\mathcal P}_n^{''} &\le \Big(\frac A2 +5\log n\Big)^{-2\ell} \theta_1^{2\ell -1}  \int_0^{\theta_1} \E \Big[  \Big| \sum_{k< e^n} X(k) r^k \sqrt{k} e^{iky} \Big|^{2\ell} \Big] dy \\
 &=  \Big(\frac A2 +5\log n\Big)^{-2\ell} \theta_1^{2\ell} \ell! \Big( \sum_{k< e^n} kr^{2k} \Big)^{\ell},
 \end{align*}
 since $\sum_{k< e^n} X(k) r^{k} \sqrt{k} e^{iky}$ is distributed like a complex Gaussian with variance $\sum_{k< e^n} kr^{2k}$.   Since $\ell! \le \ell^\ell$, $\theta_1 \le 2\pi/(ne^n)$ and $\sum_{k\le e^n} kr^{2k} \le e^{2+2n}$ (as $r\le e^{1/K}$) it follows that 
 $$ 
 {\mathcal P}_n^{''} \le \Big(\frac A2 +5\log n\Big)^{-2\ell}  \ell^\ell \Big( \frac{2\pi e}{n}\Big)^{2\ell}. 
 $$
 Upon choosing  $\ell$ to be an integer around $n(A/2+5\log n)$, we see that ${\mathcal P}_n^{''} \ll e^{-n-A}/n^{10}$; indeed ${\mathcal P}_n^{''}$ is much smaller than this, but we have just matched our earlier bound in \eqref{6.5}.  
 Combining this with \eqref{6.4} and \eqref{6.5}, it follows that ${\mathcal P}_n \ll e^{-A}/n^{9}$, which when inserted in \eqref{6.1} yields \cref{prop5.2}.  Note that the factor $1/n^{9}$ ensures that the sum in \eqref{6.1} converges, and it is for this reason that the ``safety valve" term $10\log n$ was introduced in Definition \ref{def5.1} (and one of the reasons why Lemma \ref{lem:Harper} has the flexible term $h(j)$). 
  \hfill \qed
  
%
% SECTION 7
%
\section{Proof of \cref{prop5.3}}
\label{sec:BallotSavings}

\noindent The proof of the upper bound in \cref{thm:main} has finally been reduced to \cref{prop5.3}, which we shall now obtain as an application of Lemma \ref{lem:Harper}. We focus on showing that 
 \begin{equation} 
 \label{7.1} 
 \E \Big[ \ind_{\cG_r(A,\theta;K)} |F(re^{i\theta})|^2 \Big] \asymp \frac{A}{\sqrt{\log K}} K. 
 \end{equation} 
 Since $\cG_r(A;K)$ is the event where $\cG_r(A,\theta;K)$ holds for all $\theta$, it 
 then follows that 
 $$ 
 \E \Big[ \ind_{\cG_r(A;K)} \int_0^{2\pi} |F(re^{i\theta})|^2 d\theta \Big] 
 \leq \int_0^{2\pi} \E \Big[ \ind_{\cG_r(A,\theta;K)} |F(re^{i\theta})|^2  \Big] d\theta \ll \frac{A}{\sqrt{\log K}} K, 
 $$ 
 completing the proof of the proposition. By rotational symmetry it is enough to establish \eqref{7.1} in 
 the case $\theta =0$.  
 
 Put $x_k =\text{Re}(X(k))$ so that the $x_k$ are independent real normal variables with mean $0$ and 
 variance $1/2$.   Then we may write 
 \begin{equation} 
 \label{7.2} 
 \E \Big[ \ind_{\cG_r(A,0;K)} |F(r)|^2  \Big]
= \frac{1}{\pi^{\lfloor K \rfloor/2}} \intmult_{\mathcal{R}}  \exp\Big( \sum_{k \leq K} \big( \frac{  2x_k r^k}{\sqrt{k}} - x_k^2 \big) \Big) dx_1 dx_2  \cdots dx_{\lfloor K\rfloor},
\end{equation} 
where $\mathcal{R} \subseteq \R^{\lfloor K\rfloor}$ is the region given by $(x_k)_{k=1}^{\lfloor K\rfloor} \in \R^{\lfloor K\rfloor}$ satisfying
$$
	\sum_{k < e^n} \Big(\frac{x_k r^k }{\sqrt{k}} - \frac{r^{2k}}{k} \Big) \leq A +10 \log n
$$ 
for all $1\le n \le \log K$.  Write $y_k = x_k-r^k/\sqrt{k}$, which completes the square and allows us 
to express the integral in \eqref{7.2} as 
$$
 \exp\Big({\sum_{k \leq K} \frac{r^{2k}}{k} }\Big)  \frac{1}{\pi^{\lfloor K\rfloor/2}}  \intmult_{\mathcal{R}'} \exp\Big({-\sum_{k \leq K} y_k^2 }\Big) dy_1 \cdots dy_{\lfloor K\rfloor}, 
 $$
where ${\mathcal R}' \subseteq \R^{\lfloor K\rfloor}$ is the region given by $\sum_{k< e^n} y_kr^{k}/\sqrt{k} \le A+10\log n$ for all $1\le n\le \log K$. Since $1 \leq r \leq e^{1/K}$, it follows that 
\begin{equation} 
\label{7.3} 
 \E \Big[ \ind_{\cG_r(A,0;K)} |F(r)|^2  \Big] =\exp\Big( \sum_{k \leq K} \frac{r^{2k}}{k} \Big)  \P[\cB] 
\asymp K \P[\cB] , 
\end{equation} 
where $\cB$  is the event that, for all $1 \leq n \leq  \log K$, one has
\begin{equation} 
\label{7.4} 
\sum_{k< e^n} \frac{y_kr^{k}}{\sqrt{k}} \le A+ 10\log n, 
\end{equation}
for independent normal random variables $y_k$ with mean $0$ and variance $\tfrac 12$. 

To complete the proof, we are now ready to apply \cref{lem:Harper} with the random variables $G_m = \sum_{e^{m-1} \le k< e^m} y_kr^k/\sqrt{k}$ for $1 \leq m \leq \log K$. Note that the variables $G_m$ are Gaussians with variance $\tfrac{1}{2}\sum_{e^{m-1} \leq k < e^m} r^{2k}/k$ which lies between $1/20$ and $20$ as required in \cref{lem:Harper}. Thus,  we deduce that $\P [\cB] \asymp A/\sqrt{\log K}$.  Using this estimate in  
\eqref{7.3} now finishes the proof of \eqref{7.1} in the case $\theta=0$.  This completes the proof of the proposition, and hence the upper bound of \cref{thm:main}.  \hfill \qed  \\ 

%/////////////////////////////////////////////
%	PART II - The lower bound
%/////////////////////////////////////////////
\begin{center}
{\bf Part II: The lower bound of \cref{thm:main}}
\end{center}

%
%	SECTION 8
%
\section{Deducing the lower bound from two propositions}
\label{sec:ProofLower}

Now, we turn our focus to the lower bound portion of \cref{thm:main}. If $0 \leq q \leq \frac{1}{2}$ then by H\"{o}lder's inequality  it follows that
\[
 \E \big[ |A(N)| \big] \leq \big( \E \big[ |A(N)|^{3/2} \big] \big)^{\frac{2-4q}{3-4q}} \big( \E \big[ |A(N)|^{2q} \big] \big)^{\frac{1}{3-4q}}.
\]
The upper bound $\E [ |A(N)|^{3/2} ] \ll (\log N)^{-3/4}$ in \cref{thm:main}  has already been established. 
If we knew the lower bound $\E[ |A(N)| ] \gg (\log N)^{-1/4}$, then the desired lower bound for $\E [ |A(N)|^{2q} ]$ would follow. 
Thus, it is enough to prove the lower bound in the range $\frac{1}{2} \leq q \leq 1$. 

Here we shall reduce the lower bound part of the main theorem to propositions involving bounds for the second moment of $F_K(z)$ defined by \eqref{2.1} for any real number $K \geq 1$. 

\begin{proposition} 
	\label{prop8.1}
	For $\frac{1}{2} \leq q \leq1$ and $0 <  r < 1$ we have
	$$
	 \E[ |A(N)|^{2q} ] \geq \frac{1}{4}\Big( \E\Big[ \Big( \frac{1}{2\pi N} \int_0^{2\pi}| F_{N/2}(re^{i\theta})|^2 d\theta   \Big)^q \Big] -  \E\Big[ \Big( \frac{ r^{N}}{2\pi N}   \int_0^{2\pi}| F_{N/2}(e^{i\theta})|^2 d\theta \Big)^q \Big]  \Big).
	$$
\end{proposition}

\begin{proposition} 
\label{prop8.2}
Let $K\ge 10$ be a real number, and let $F(z)= F_K(z)$ be as in \eqref{2.3}.  Uniformly for $\frac{1}{2} \le q \le 1$ and $e^{-1/40} \leq  r < 1$ we have 
$$ 
\E\Big[ \Big(\int_0^{2\pi} |F_K(re^{i\theta})|^2 d\theta \Big)^q \Big] \gg \Big( \frac{K_r}{(1-q)\sqrt{\log K_r} + 1} \Big)^q,
$$
where  $\log K_r$ is the largest integer such that $K_r \leq  \min\{ \frac{-1}{4\log r}, K \}$.
\end{proposition}

With $r= e^{-N/V}$ for a suitably large, but fixed, constant $V$, \cref{prop8.2} gives 
$$ 
\E\Big[ \Big( \frac{1}{2\pi N} \int_0^{2\pi}| F_{N/2}(re^{i\theta})|^2 d\theta   \Big)^q \Big] 
\gg \Big( \frac{1}{V (1+(1-q)\sqrt{\log N})}\Big)^q, 
$$ 
while \cref{prop3.2} gives 
$$ 
\E\Big[ \Big( \frac{ r^{N}}{2\pi N}   \int_0^{2\pi}| F_{N/2}(e^{i\theta})|^2 d\theta \Big)^q \Big]  \ll \Big(\frac{e^{-V}}{1+(1-q)\sqrt{\log N}}\Big)^q.
$$
Combining these estimates with \cref{prop8.1}, and choosing $V$ to be a large enough constant, we obtain
 the lower bound in \cref{thm:main}  in the range $\frac 12 \le q\le 1$.  
  
%
%	SECTION 9
%
\section{Proof of \cref{prop8.1}}

As with the upper bound in \cref{sec:DecomposeUpper}, we begin by decomposing the definition of $A(N)$ in terms of $a(\lambda)$ for partitions $\lambda$ (see \eqref{2.1} and \eqref{2.2}), grouping terms according to the size of the largest part $\lambda_1$.  
Define
\[
A_1(N) =  \sum_{\substack{|\lambda|=N \\ N/2 < \lambda_1 \leq N}} a(\lambda) =\sum_{N/2 < n\le N} \frac{X(n)}{\sqrt{n}} A(N-n) 
\quad \text{and} \quad 
\widetilde{A}_1(N) = \sum_{\substack{|\lambda|=N \\ \lambda_1 \leq N/2}} a(\lambda), 
\]
so that $A(N) = A_1(N) +\widetilde{A}_1(N)$.  Using that $X(n)$ is distributed identically to $-X(n)$ for $N/2<n\le N$, we see that 
\begin{equation} 
\label{9.1}  
\E [|A(N)|^{2q}] = \frac 12 \E\Big [ |A_1(N) + \widetilde{A}_1(N)|^{2q} + |- A_1(N) + \widetilde{A}_1(N)|^{2q} \Big] \ge \frac 12 \E [ |A_1(N)|^{2q}],  
\end{equation} 
where the last inequality holds because $\max( |-z+w|, |z+w| ) \ge |z|$ for any two complex numbers $z$ and $w$.  
 
To obtain a lower bound on $\E[|A_1(N)|^{2q}]$, we first condition on the variables $X(k)$ for all $k \le N/2$.   Note that $A(N-n)$ is then determined for all $N/2 < n \le N$.   Therefore 
$$ 
A_1(N) = \sum_{N/2 < n\le N} \frac{X(n)}{\sqrt{n}} A(N-n), 
$$ 
being a linear combination of the independent standard complex Gaussians $X(n)$, is distributed like a complex Gaussian with mean $0$ and variance $\sum_{N/2 <n \le N} |A(N-n)|^2/n$.  With $\E_1$ denoting the conditional expectation (fixing $X(k)$ for $k\le N/2$), it follows that 
$$
\E_1 \Big[ |A_1(N)|^{2q} \Big]  = C_q \Big( \sum_{N/2 < n \le N} \frac{|A(N-n)|^2}{n} \Big)^q,  
$$ 
where $C_q = \E [ |Z|^{2q}]$ is the $2q$-th moment of a standard complex Gaussian $Z$ with mean $0$ and variance $1$.  H{\" o}lder's inequality gives 
$$
\E [ |Z|^{2q} ] \ge \frac{(\E [ |Z|^2])^{2-q}}{( \E [ |Z|^4 ])^{1-q}} = \frac{1}{2^{1-q}},
$$ 
 and so we obtain 
 $$ 
 \E_1 \Big[ |A_1(N)|^{2q} \Big]  \ge \frac{1}{2^{1-q}}  \Big( \sum_{N/2 < n \le N} \frac{|A(N-n)|^2}{n} \Big)^q \ge \frac 12 \Big( \frac 1N \sum_{n<N/2} |A(n)|^2 \Big)^q.
 $$
Now taking the full expectation and using \eqref{9.1}, we deduce that
\begin{equation} \label{9.2} 
\E\big[|A(N)|^{2q}\big] \geq  \frac{1}{4} \E \Big[ \Big( \frac{1}{N} \sum_{n< N/2} |{A}(n)|^2\Big)^q \Big].
\end{equation} 

Write 
$$ 
F_{N/2}(z) = \exp\Big( \sum_{k\le N/2} \frac{X(k)}{\sqrt{k}} z^k\Big) = \sum_{n=0}^{\infty} {\widetilde A}_1(n) z^n, 
$$ 
and note that ${\widetilde A}_1(n) = A(n)$ for $n\le N/2$.   Note that Parseval's identity gives, for any $0< r \le 1$, 
\begin{align*}
\sum_{n < N/2} |A(n)|^2 &= \sum_{n< N/2} |\widetilde{A}_1(n)|^2 \ge \sum_{n=0}^{\infty} |\widetilde{A}_1(n)|^2 r^{2n} - r^{N} 
\sum_{n=0}^{\infty} |\widetilde{A}_1(n)|^2 
\\
&= \frac{1}{2\pi} \int_0^{2\pi} |F_{N/2} (re^{i\theta})|^2 d\theta - \frac{r^{N}}{2\pi} \int_0^{2\pi} |F_{N/2}(e^{i\theta})|^2 d\theta.
\end{align*} 
Since $|z+w|^q \le |z|^q + |w|^q$ for $q\le 1$, it follows that 
$$
\Big(\sum_{n < N/2} |A(n)|^2 \Big)^q \ge\Big( \frac{1}{2\pi} \int_0^{2\pi} |F_{N/2} (re^{i\theta})|^2 d\theta\Big)^q - \Big(  \frac{r^{N}}{2\pi} \int_0^{2\pi} |F_{N/2}(e^{i\theta})|^2 d\theta\Big)^q, 
$$ 
and inserting this in \eqref{9.2}, we obtain  \cref{prop8.1}. \hfill \qed
 
%
% SECTION 10
%
\section{Plan for the proof of \cref{prop8.2}}

The proof of the lower bound in \cref{thm:main} has now been reduced to establishing \cref{prop8.2}.  Let ${\mathcal L} ={\mathcal L}(X)$ denote any (random) subset of $\theta \in [0,2\pi)$ with the random subset depending possibly on the instantiation of the random variables $X$.   Using H{\" o}lder's inequality, we obtain 
\begin{align} 
\label {10.1} 
\E \Big[ \Big( \frac{1}{2\pi} \int_0^{2\pi} |F(re^{i\theta})|^2 d\theta\Big)^q \Big] &\ge \E \Big[ \Big( \frac{1}{2\pi} \int_{{\mathcal L}(X)} |F(re^{i\theta})|^2 d\theta\Big)^q \Big] \nonumber \\ 
& \geq \frac{\displaystyle \Big( \E\Big[ \frac{1}{2\pi} \int_{\cL(X)} |F(re^{i\theta})|^2 d\theta \Big] \Big)^{2-q} }{\displaystyle \Big( \E\Big[  \Big(\frac{1}{2\pi} \int_{\cL(X)} |F(re^{i\theta})|^2 d\theta \Big)^2 \Big] \Big)^{1-q} }.
\end{align}
We apply this idea to a carefully chosen random subset ${\mathcal L}(X)$ where the second and fourth moments will be of comparable size so that there is no loss involved in applying H{\" o}lder's inequality in \eqref{10.1}.   The random set ${\mathcal L}(X)$ is defined similarly to   \cref{def5.1}, keeping once again the ballot problem in mind.  

\begin{definition}
	  \label{def10.1} Let $K \geq 10$. Suppose $e^{-1/40} \leq r < 1$ and define $\log K_r$ to be the largest integer such that $K_r \leq \min\big\{ \frac{-1}{4\log r}, K\big\}$.  Let $A$ be a real number with $1\le A \le \sqrt{\log K_r}$. Define $\cL(\theta) = \cL_r(A,\theta;K)$ to be the  following event:  For each $1\le n \le \log K_r$ one has 
\[
 \sum_{k< e^n} \Big( \mathrm{Re} \frac{X(k) r^k e^{ik\theta}}{\sqrt{k}}  - \frac{r^{2k}}{k} \Big) \le A - 5\log n. 
\]
 Define $\cL = \cL_r(A; K)$ to be the random subset of $\theta \in [0,2\pi]$ such that $\cL(\theta)$ holds.   
\end{definition}

With this choice of the random subset ${\mathcal L}$, we seek a lower bound for the numerator in \eqref{10.1} and an upper bound for the denominator there.  We start with the easier case of the lower bound.   Since ${\mathcal L}$ denotes the subset of $\theta$ for which ${\mathcal L}(\theta)$ holds, we find that 
$$ 
\E \Big[ \frac{1}{2\pi} \int_{\mathcal L} |F(re^{i\theta})|^2 d\theta \Big] = \E \Big[ \frac{1}{2\pi} \int_0^{2\pi} \ind_{\cL(\theta)} |F(re^{i\theta})|^2 d\theta \Big] = \frac 1{2\pi} \int_0^{2\pi}  \E \big[\ind_{\cL(\theta)} |F(re^{i\theta})|^2 \big] d\theta, 
$$
and by the rotational symmetry of the random variables $X$, this equals
$$
\E \big[ \ind_{\cL(0)} |F(r)|^2 \big].
$$
 Arguing as in \cref{sec:BallotSavings} we may see that 
 \begin{equation*}
 \E\big[ \ind_{\cL(0)} |F(r)|^2 \big] = \exp\Big( \sum_{k \leq K} \frac{r^{2k}}{k} \Big) \mathbb{P}[\cB],
 \label{eqn:shift-probability-2}
 \end{equation*}	
 where $\cB$ is the event that, for all $1 \leq n \leq \log K_r$, one has
\[
\sum_{k < e^n} \frac{y_k r^k}{\sqrt{k}} \leq A - 5 \log n,
\]
for independent normal random variables $y_k$ with mean $0$ and variance $\frac{1}{2}$.   We now invoke \cref{lem:Harper} with the random variables $G_m = \sum_{e^{m-1} \leq k < e^{m}} y_k r^{k}/\sqrt{k}$ for $1 \le m\le \log K_r$.   Note that these variables $G_m$ are Gaussian with variance $\frac{1}{2}\sum_{e^{m-1} \leq k< e^{m}} r^{2k}/k$ which lies between $1/20$ and $20$ as required in \cref{lem:Harper}.  Therefore it follows that $\mathbb{P}(\cB) \asymp A/\sqrt{\log K_r}$, and we conclude that 
\begin{equation} 
\label{10.2} 
\E \Big[ \frac{1}{2\pi} \int_{\mathcal L} |F(re^{i\theta})|^2 d\theta \Big]  \gg \frac{A}{\sqrt{\log K_r}} \exp\Big( \sum_{k \leq K} \frac{r^{2k}}{k} \Big) 
\gg \frac{A K_r}{\sqrt{\log K_r}}. 
\end{equation}

Now we turn to the harder problem of obtaining satisfactory upper bounds for the denominator in \eqref{10.1}.   Expanding out we see that 
$$
\E \Big[ \Big(\frac{1}{2\pi} \int_{\cL} |F(re^{i\theta})|^2 d\theta \Big)^2 \Big] = \frac{1}{(2\pi)^2} \E \Big[ \int_{0}^{2\pi} \int_0^{2\pi} 
\ind_{\cL(\theta_1)} |F(re^{i\theta_1})|^2 \ind_{\cL(\theta_2)}|F(re^{i\theta_2})|^2 d\theta_1 d\theta_2 \Big], 
$$
and upon writing $\theta = \theta_2 -\theta_1$ (and taking $\theta$ to be in $[-\pi, \pi)$) and using rotational symmetry this equals 
$$ 
\frac{1}{2\pi} \int_{-\pi}^{\pi} \E \big[ \ind_{\cL(0)} |F(r)|^2 \ind_{\cL(\theta)} |F(re^{i\theta})|^2 \big] d\theta. 
$$

\begin{proposition}	\label{prop10.2}
With notations as above, and any $\theta \in [-\pi, \pi)$ we have
$$ 
\E  \big[ \ind_{\cL(0)} |F(r)|^2 \ind_{\cL(\theta)} |F(re^{i\theta})|^2 \big]   \ll A^2 e^{2A} \frac{K_r^2}{\log K_r} \frac{\min (K_r,2\pi/|\theta|)}{(\log \min (K_r, 2\pi/|\theta|))^8}. 
$$ 
\end{proposition}

We postpone the proof of \cref{prop10.2} to the next section, and assuming this bound now finish our proof of 
\cref{prop8.2}.   Applying \cref{prop10.2} we obtain 
$$ 
\E \Big[ \Big(\frac{1}{2\pi} \int_{\cL} |F(re^{i\theta})|^2 d\theta \Big)^2 \Big]  \ll A^2 e^{2A} \frac{K_r^2}{\log K_r} \int_{-\pi}^{\pi} \frac{\min (K_r,2\pi/|\theta|)}{(\log \min (K_r, 2\pi/|\theta|))^8} d\theta \ll A^2 e^{2A} \frac{K_r^2}{\log K_r}. 
$$ 
Using this upper bound for the denominator in \eqref{10.1} together with the lower bound for the numerator (given in \eqref{10.2}) we conclude that 
$$ 
\E\Big[ \Big( \frac{1}{2\pi} \int_0^{2\pi} |F(re^{i\theta})|^2 d\theta \Big)^q\Big] \gg e^{-2A(1-q)} \Big( \frac{A K_r}{\sqrt{\log K_r}}\Big)^{q}.
$$ 
Selecting $A = \sqrt{\log K_r}/( (1-q) \sqrt{\log K_r} + 1)$ completes the proof of \cref{prop8.2}.

%
% SECTION 11
%
\section{Proof of \cref{prop10.2}}
\label{sec:TiltedProbability}

%The log-correlated structure of the random variables $\log |F(re^{i\theta})| $  now takes center stage.  We shall break up the variables   $\log |F(r)|$ and $\log|F(re^{i\theta})|$ into two pieces: a strongly correlated part (small $k$) and a weakly correlated part (large $k$). Restricted to the event $\mathcal{L}(0) \cap \mathcal{L}(\theta)$, the strongly correlated part will be large but not too large and the weakly correlated part will produce the desired savings. 

 Given $\theta \in [-\pi, \pi)$, define $M = M(r,\theta)$ to be the smallest integer such that 
\begin{equation}
e^M \ge \min\{ 10^3/|\theta|, K_r/e \}.
\label{11.1}
\end{equation}
Note that $\log K_r \ge 2$ in \cref{def10.1}, and so $M \ge 1$.
Set 
\begin{equation} 
\label{11.2} 
A_0(M) = \text{Re}  \sum_{k< e^M}\Big( \frac{X(k)r^k}{\sqrt{k}} - \frac{r^{2k}}{k}\Big), \qquad A_\theta(M) = \text{Re}  \sum_{k< e^M}\Big( \frac{X(k)r^ke^{ik\theta}}{\sqrt{k}} - \frac{r^{2k}}{k}\Big), 
\end{equation} 
and define for $M+1 \le m \leq \log K_r$, 
\begin{equation} 
\label{11.3} 
Z_0(m) = \text{Re} \sum_{e^{m-1} \leq k < e^{m}}\ \frac{X(k)}{\sqrt{k}} r^k ,
\qquad Z_\theta(m) = \text{Re} \sum_{e^{m-1} \leq k < e^{m}}  \frac{X(k)}{\sqrt{k}} r^ke^{ik\theta}. 
\end{equation}

Our goal is to bound the expected value of $|F(r)|^2 |F(re^{i\theta})|^2$ when restricted to the event ${\cL}(0) \cap \cL(\theta)$.  
Recall that $\cL(0) \cap \cL(\theta)$ is the event satisfying the inequalities (for $1\le n\le \log K_r$) 
$$ 
\sum_{k< e^n}  \Big( \text{Re } \frac{X(k)}{\sqrt{k}} r^k - \frac{r^{2k}}{k} \Big) \le A-5 \log n, \qquad 
\sum_{k< e^n}  \Big( \text{Re} \frac{X(k)}{\sqrt{k}} r^ke^{ik\theta}  -\frac{r^{2k}}{k}\Big)\le A-5 \log n. 
$$
 Since our goal is to obtain upper bounds, we replace the event $\cL(0) \cap \cL(\theta)$ with a less restrictive event which is easier to handle.  
This less restrictive event $\widetilde{\cL}$ is defined by the constraints 
\begin{equation} 
\label{11.4} 
A_0(M), A_\theta(M) \le A -5 \log M, 
\end{equation} 
together with, for $M+1 \leq m\le \log K_r$ 
\begin{equation} 
\label{11.5} 
\sum_{e^M \leq k < e^m} \Big( \text{Re }\frac{X(k)}{\sqrt{k}} r^k - \frac{r^{2k}}{k}\Big), \   \sum_{e^{M} \leq k < e^m} \Big( 
\text{Re } \frac{X(k)}{\sqrt{k}} r^k e^{ik\theta} - \frac{r^{2k}}{k} \Big) \le A - \min ( A_0(M), A_\theta(M),0). 
\end{equation}

Before entering into the details, let us give a loose description of the argument.  The values $k$ below $e^M$ are thought of as small, and here $e^{ik\theta}$ may be thought of as close to $1$.  The constraints imposed by ${\cL}(0)$ and $\cL(\theta)$ are strongly correlated for such $k$, and so are the quantities $A_0(M)$ and $A_\theta(M)$.   The ``barrier events" in \eqref{11.4} prevent the contribution of these small $k$ from getting too large.  In the range $e^M \le k \le K_r$, the oscillation of $e^{ik\theta}$ becomes significant, and the terms $Z_0(m)$ and $Z_\theta(m)$ behave almost independently of each other.  This allows us to think of the constraints in \eqref{11.5} as corresponding to two independent applications of the ballot problem, leading eventually to the saving of $\log K_r$ in Proposition \ref{prop10.2}. 
Lastly the terms with $K \ge k> K_r$ contribute a negligible amount as they are weighted down by the factor $r^k$ which is small in this range.

Returning to the proof, in the notation just introduced, we have 
\begin{align*}
\E \big[ \ind_{\cL(0)\cap \cL(\theta)} &|F(r)F(re^{i\theta})|^2\big]
\le \exp\Big( 4\sum_{k < e^M} \frac{r^{2k}}{k} \Big) \E\Big[ \ind_{\widetilde{\cL}} 
\exp( 2A_0(M) + 2A_{\theta}(M)) \\
&\prod_{M+1\le m\le \log K_r}  \exp( 2Z_0(m) +2 Z_\theta(m)) \exp\Big(2  \sum_{K_r \leq k\le K} \frac{r^{k}}{\sqrt{k}} 
\text{Re }(X(k) + X(k) e^{ik\theta}) \Big) \Big]. 
\end{align*}
We make a few initial simplifications to this quantity, before getting to the crux of the proof.  Note that the terms involving $X(k)$ with $K_r\leq k\le K$ are independent of the random variables with $k< K_r$, and are not constrained by \eqref{11.4} or \eqref{11.5}.   So we may separate these terms from our expression above, and they contribute 
$$  
\le \frac 12 \E \Big[ \exp\Big( 4 \sum_{K_r \leq k\le K} \text{Re} X(k) \frac{r^{k}}{\sqrt{k}} \Big) + 
 \exp\Big( 4 \sum_{K_r \leq k\le K} \text{Re} (X(k)e^{ik\theta})  \frac{r^{k}}{\sqrt{k}} \Big)\Big] 
 =\exp\Big( 4 \sum_{K_r \leq k \le K} \frac{r^{2k}}{k} \Big), 
 $$
since $\sum_{K_r \leq k \le K} \text{Re} (X(k)) r^{k}{\sqrt{k}}$ and $\sum_{K_r \leq k \le K} \text{Re }(X(k)e^{ik\theta}) r^{k}/\sqrt{k}$ are distributed like Gaussian random variables with mean $0$ and variance $\frac 12 \sum_{K_r \leq k\le K} r^{2k}/k$ (compare with \cref{lem3}).   Noting that 
$$ 
\sum_{k< e^M} \frac{r^{2k}}{k} \le \sum_{k< e^{M}} \frac 1k = M + O(1), \qquad \text{and } \sum_{K_r \leq k\le K} \frac{r^{2k}}{k} = O(1),
$$ 
we conclude that 
\begin{align} 
\label{11.6} 
\E \big[ \ind_{\cL(0)\cap \cL(\theta)} &|F(r)F(re^{i\theta})|^2\big]  \nonumber\\
&\ll e^{4M} \E \Big[ \ind_{\widetilde{\cL}} 
\exp( 2A_0(M) + 2A_{\theta}(M)) 
\prod_{M+1 \le m\le \log K_r}  \exp( 2Z_0(m) +2 Z_\theta(m)) \Big]. 
\end{align} 

We now state a proposition (to be proved in the next section) which amounts to two applications of the ballot problem, and granting this proposition, we will be able to finish the proof of \cref{prop10.2}.

\begin{proposition} 
\label{prop11.1}  Keep notations as above.  Given a real number $B$, let ${\mathcal E}$ denote the following event: for all $M+1 \le m \le \log K_r$ one has 
\begin{equation} 
\label{11.7} 
\sum_{e^M  \leq k < e^m} \Big( \mathrm{Re }\frac{X(k)}{\sqrt{k}} r^k - \frac{r^{2k}}{k}\Big), \ \  \sum_{e^{M} \leq k < e^m} \Big( 
\mathrm{Re } \frac{X(k)}{\sqrt{k}} r^k e^{ik\theta} - \frac{r^{2k}}{k} \Big) \le B. 
\end{equation}
Then 
$$ 
\E \Big[ \ind_{{\mathcal E}} \prod_{M+1 \le m \le \log K_r} \exp(2 Z_0(m)+ 2 Z_\theta(m)) \Big] 
\ll \frac{K_r^2}{e^{2M}} \Big( \frac{1+ \max(0, B)}{\sqrt{1+\log (K_r/e^M)}}\Big)^2. 
$$ 
\end{proposition} 

Assuming this proposition, we now resume the proof of \cref{prop10.2}, starting from \eqref{11.6}.  Applying \cref{prop11.1} with $B= A - \min( A_0(M), A_\theta(M),0)$ we obtain 
\begin{align*}
 \E \big[ &\ind_{\cL(0)\cap \cL(\theta)} |F(r)F(re^{i\theta})|^2\big] 
\nonumber \\
&\ll  \frac{K_r^2 e^{2M}}{1+\log(K_r/e^M)}  \E \Big[ \ind_{\widetilde{\cL}} \exp( 2A_0(M) + 2A_{\theta}(M)) \big(A + \max(-A_0(M), -A_\theta(M),0) \big)^2\Big]. 
\end{align*}
Here we have abused notation a little, and the event $\widetilde{\cL}$ refers now only to the constraint \eqref{11.4} on the variables $X(k)$ with $k< e^M$. Notice also that we have  used \cref{prop11.1} treating $X(k)$ for $k < e^M$ as fixed. By rotational symmetry, we may assume that $A_0(M) \le A_\theta(M)$; the other case contributes an identical amount.  Then bounding $A_\theta(M)$ by $A- 5 \log M$, we conclude that 
\begin{align} 
\label{11.8} 
\E \big[ &\ind_{\cL(0)\cap \cL(\theta)} |F(r)F(re^{i\theta})|^2\big] 
\nonumber \\
&\ll  \frac{K_r^2 e^{2M}}{1+\log(K_r/e^M)} \frac{e^{2A}}{M^{10}}  \E \Big[ \ind_{A_0(M)\le A-5\log M} \exp( 2A_0(M))  \big( A+ \max(-A_0(M),0) \big)^2\Big]. 
\end{align}

Now $\text{Re }\sum_{k< e^M} X(k) r^k/\sqrt{k}$ is a real Gaussian with mean $0$ and variance $\frac 12 \sum_{k< e^M} r^{2k}/{k}$.  Therefore, using also that $\sum_{k< e^M} r^{2k}/k = M + O(1)$, 
\begin{align*}
 \E \Big[ &\ind_{A_0(M)\le A-5\log M} \exp( 2A_0(M)) \big( A+\max(-A_0(M),0)\big)^2\Big] \\
 & = 
 \frac{1}{\sqrt{\pi \sum_{k< e^M} r^{2k}/k}} \int_{-\infty}^{A-5\log M} e^{2x} (A+ \max(-x,0))^2 \exp\Big(-\frac{(x+\sum_{k< e^M}r^{2k}/k)^2}{\sum_{k< e^M} r^{2k}/k} \Big) dx \\ 
 &\ll (A^2+M) e^{-M}. 
 \end{align*}
Inserting this in \eqref{11.8}, we conclude that 
$$ 
\E \big[ \ind_{\cL(0)\cap \cL(\theta)} |F(r)F(re^{i\theta})|^2\big]  \ll 
\frac{K_r^2 e^M}{1+\log (K_r/e^M)} \frac{e^{2A}}{M^{10}} (A^2+M) \ll \frac{K_r^2 e^M}{1+\log (K_r/e^M)} \frac{A^2 e^{2A}}{M^9}. 
$$
Upon recalling the definition of $M$ in \eqref{11.1}, and noting that $M(1 +\log (K_r/e^M)) \ge \log K_r$, this completes the proof of \cref{prop10.2}.  \hfill \qed

%
% SECTION 12
%
\section{Proof of \cref{prop11.1}} 
 
The proof of the lower bound for \cref{thm:main} has been reduced to \cref{prop11.1}. Here we are focussing on the range $e^M \le K \le K_r$ where there is substantial oscillation in the terms $e^{ik\theta}$, and we shall see that the variables $Z_0(m)$ and $Z_{\theta}(m)$ for $M+1 \leq m \leq \log K_r$ are largely uncorrelated (or more precisely, very weakly correlated). By exploiting properties of bivariate Gaussian vectors, these weakly correlated Gaussians may be replaced with \textit{independent} Gaussians. Thus, the expectation in \cref{prop11.1} will essentially  split into two independent Gaussian random walks where an analysis similar to \cref{sec:BallotSavings}  will ultimately carry over. 
 
We first dispense with the case when $\log (K_r/e^M) \le 10$.  Note that 
\begin{align*}
\E \Big[ \ind_{\mathcal E} &\prod_{M+1 \le m\le \log K_r} \exp(2Z_0(m) +2Z_\theta(m)) \Big] \\
&\le \frac 12 \Big( \E\Big[  \prod_{M+1 \le m\le \log K_r} \exp(4Z_0(m)) \Big] + \E\Big[  \prod_{M+1 \le m\le \log K_r} \exp(4Z_\theta(m))\Big]\Big) \\ 
&=  \E\Big[  \prod_{M+1 \le m\le \log K_r} \exp(4Z_0(m)) \Big] = \E \Big[ \exp\Big( 4\text{Re }\sum_{e^M \le k < K_r} \frac{X(k)r^{k}}{\sqrt{k}}\Big)\Big],
\end{align*}
by rotational symmetry.  Now $\sum_{e^M \le k <K_r} \text{Re }(X(k)r^k/\sqrt{k})$ is distributed like a Gaussian random variable with mean $0$ and variance $\frac 12 \sum_{e^M \le k <K_r} r^{2k}/k$, and this variance is bounded by our assumption that $\log (K_r/e^M) \le 10$.  It follows that in this case, our desired quantity is bounded by an absolute constant, and the  proposition follows at once.  

Therefore we may assume that $K_r > e^{M+10}$ below.  Upon recalling the definition of $M$ (see \eqref{11.1}) we may thus assume that $\theta \in (-\pi, \pi]$ satisfies $10^3/|\theta| > K_r/e$ and that 
\[
e^M |\theta| \ge 10^3. 
\]
With this easy case out of the way, we now embark on the proof proper. 

For $M+1\le m\le \log K_r$ note that the variables $Z_0(m)$ and $Z_\theta(m)$ depend only on $X(k)$ for $e^{m-1} \le k < e^{m}$, and so these variables $Z_0(m)$ and $Z_\theta(m)$ are independent for different values of $m$.  We therefore begin by discussing, for a given $M+1\le m\le \log K_r$, the probability that $Z_0(m)$ and $Z_\theta(m)$ satisfy some event, and then by combining those results for different $m$ we will obtain \cref{prop11.1}.

We recall that a pair of real random variables $(Y_1, Y_2)$ is said to have a bivariate normal distribution if every linear combination $a_1 Y_1 + a_2 Y_2$ with $a_1, a_2 \in {\Bbb R}$ is a univariate normal random variable.  The bivariate normal distribution is determined by the means $\mu_1 =\E[Y_1]$ and $\mu_2 = \E [Y_2]$ together with the $2\times 2$ (symmetric) covariance matrix $\E[Y_i Y_j]$ for $1\le i, j\le 2$.  Denote by $\sigma_i^2 = \E [Y_i^2]$ for $i=1$, $2$, and by $\rho \sigma_1 \sigma_2$ the covariance $\E [Y_1Y_2]$ so that $|\rho| \le 1$.  For a bivariate normal vector $(Y_1, Y_2)$ with these parameters, the probability density at $(x_1, x_2) \in {\Bbb R}^2$ is given by 
\begin{equation} 
\label{12.1} 
\small
\frac{1}{2\pi \sigma_1 \sigma_2 \sqrt{1-\rho^2}} \exp\Big( - \frac{1}{2(1-\rho^2)} \Big( \Big(\frac{x_1-\mu_1}{\sigma_1}\Big)^2 -2 \rho \Big(\frac{x_1-\mu_1}{\sigma_1} \Big) \Big( \frac{x_2-\mu_2}{\sigma_2}\Big) + \Big( \frac{x_2 -\mu_2}{\sigma_2}\Big)^2 \Big)\Big). 
\end{equation} 
Here we ignore the degenerate case when $|\rho|=1$.

If $(Y_1,Y_2)$ is a bivariate normal vector, then in general $Y_1$ and $Y_2$ need not be independent, and indeed the case when they are independent corresponds to requiring the covariance $\E[Y_1Y_2]$ to be $0$ (equivalently $\rho=0$ above).  We next observe that even in the general case, we can upper bound the probability density in \eqref{12.1} by replacing $(Y_1,Y_2)$ by a suitable pair of independent normal variables.  This is especially useful when the covariance parameter $\rho$ is small, which will be the case for us.  

Suppose $(Y_1,Y_2)$ is a bivariate normal vector, with probability density as in \eqref{2.1}.   Since 
$$ 
\Big| 2\rho \Big( \frac{x_1-\mu_1}{\sigma_1}\Big) \Big( \frac{x_2-\mu_2}{\sigma_2}\Big)\Big| 
\le |\rho| \Big(  \Big( \frac{x_1-\mu_1}{\sigma_1}\Big)^2+ \Big( \frac{x_2-\mu_2}{\sigma_2}\Big)^2\Big), 
$$ 
we see that the probability density in \eqref{12.1} is bounded above by 
\begin{equation} 
\label{12.2} 
\frac{\sqrt{1+|\rho|}}{\sqrt{1-|\rho|}} \frac{1}{2\pi \sigma_1\sigma_2 (1+|\rho|)} \exp\Big(-\frac 1{2(1+|\rho|)} \Big( \Big(\frac{x_1-\mu_1}{\sigma_1}\Big)^2 + \Big(\frac{x_2-\mu_2}{\sigma_2}\Big)^2 \Big). 
\end{equation}
Apart from the factor $\sqrt{(1+|\rho|)/(1-|\rho|)}$, the quantity above is the probability density of a pair of independent normal variables 
$(\widetilde{Y_1},\widetilde{Y_2})$ with means $\mu_1$, $\mu_2$, and variances $\sigma_1^2 (1+|\rho|)$, $\sigma_2^2(1+|\rho|)$.  Thus given any event ${\mathcal B}$ (thought of as a Borel measurable subset of ${\mathbb R}^2$) we have 
\begin{equation} 
\label{12.3} 
\P[(Y_1,Y_2) \in {\mathcal B}] \le \frac{\sqrt{1+|\rho|}}{\sqrt{1-|\rho|}} \P [(\widetilde{Y}_1,\widetilde{Y}_2) \in {\mathcal B}]. 
\end{equation}

With this preliminary discussion in place, we are now ready to handle \cref{prop11.1}.  Since Re$(X(k))$ and Im$(X(k))$ are independent normal variables, it follows that 
$$
(\text{Re}(X(k), \text{Re}(X(k)e^{ik\theta})) = (\text{Re}(X(k)), \cos(k\theta) \text{Re}(X(k)) - \sin(k\theta) \text{Im}(X(k)))
$$ is a bivariate normal vector.  Being a linear combination of independent such vectors we see that $(Z_0(m), Z_\theta(m))$ is also a bivariate normal vector.  Note that both $Z_0(m)$ and $Z_\theta(m)$ have mean $0$, variance 
\begin{equation} 
\label{12.4} 
\E[|Z_0(m)|^2] = \E[|Z_\theta(m)|^2] = \sigma_m^2 = \sum_{e^{m-1} \leq k < e^{m}} \frac{r^{2k}}{2k}
\end{equation} 
and covariance 
\begin{equation} 
\label{12.5} 
\E[Z_0(m) Z_\theta(m)] = \rho_m(\theta) \sigma_m^2 = \sum_{e^{m-1} \leq k < e^{m}} \frac{r^{2k}}{2k} \cos(k\theta). 
\end{equation} 

%Recall $F(z) = F_K(z)$ is defined by \eqref{2.1},  and $\cL = \cL_r(A;K)$ is given by \cref{def10.1} with $e^{-1/4} \leq r < 1$, $K_r = \min\{ \frac{-1}{4 \log r}, \frac{K}{e}\}$, and $1 \leq A \leq \sqrt{\log K_r}$.  Upon using rotation invariance, we have that
%\begin{equation} 
%\label{7.1} 
%\E \Big[\Big( \int_{\mathcal L} |F(re^{i\theta})|^2 d\theta\Big)^2 \Big] = 2\pi \int_{-\pi}^{\pi} \E\big[ \ind_{\cL(0)} |F(r)|^2 \ind_{\cL(\theta)} |F(re^{i\theta})|^2 \big] d\theta,
%\end{equation} 
   
In the range $M+1 \leq m \leq \log K_r$, 
\begin{equation}
\label{12.6}
\frac{1}{4} \leq \sum_{e^{m-1} \leq k \le e^{m}} \frac{1}{4k}  \leq \sigma_m^2 \leq \sum_{e^{m-1} \leq k < e^{m}} \frac{1}{2k} \leq \frac{1}{2} + \frac{1}{2e^{m-1}}	. 
\end{equation}
because $\frac{1}{2} \leq r^{2k} \leq 1$ for $k \leq K_r  $. Further 
the covariance satisfies 
\begin{equation}
\label{12.7}
|\rho_m(\theta) \sigma_m^2| = \Big| \sum_{e^{m-1} \leq k < e^{m}}r^{2k} \frac{\cos(k\theta)}{2k} \Big| =\Big|  \int_0^r \sum_{e^{m-1} \le k < e^m} t^{2k-1} e^{ik\theta} dt\Big|  
\le \frac{\pi}{|\theta| e^{m-1}}, 
\end{equation}
since by summing the geometric series, and using $|1-te^{i\theta}| \ge |\sin (\theta/2)| \ge |\theta|/\pi$ (for all $0\le t\le 1$ and $\theta \in (-\pi, \pi]$),   we may see that 
$$ 
\Big|  \sum_{e^{m-1} \le k < e^m} t^{2k-1} e^{ik\theta} \Big| \le \frac{2t^{2\lceil e^{m-1}\rceil -1}}{|1- te^{i\theta}|} \le \frac{\pi}{|\theta|}  2t^{2\lceil e^{m-1}\rceil -1}.
$$ 
Thus in this range $|\rho_m| \le 4\pi/(|\theta| e^{m-1})$ is small, being always $\le 4\pi/10^3 < 1/10$.  
 
Now define independent normal random variables $\widetilde{Z}_0(m)$ and $\widetilde{Z}_\theta(m)$ distributed identically with mean $0$ and variance $\sigma_m^2 (1+ |\rho_m(\theta)|)$.   As noted in \eqref{12.2} and \eqref{12.3} we obtain that for any event ${\mathcal B}$ (thought of as a  Borel measurable subset of ${\Bbb R}^2$) we have 
\begin{equation*} 
\E \Big[ \ind_{\mathcal B} \exp( 2 (Z_0(m) + Z_\theta(m))\Big] \le 
\frac{\sqrt{1+|\rho_m(\theta)|}}{\sqrt{1-|\rho_m(\theta)|}} \E \Big[ \ind_{\mathcal B} \exp( 2(\widetilde{Z}_0(m) + \widetilde{Z}_\theta(m))\Big]. 
\end{equation*} 
Applying this to all $M+1 \le m \le \log K_r$ (and recalling that $Z_0(m)$ and $Z_\theta(m)$ are independent for different values of $m$) we conclude that 
\begin{align} \label{12.8} 
\E \Big[ \ind_{\mathcal E} \prod_{M+1\le m\le \log K_r} \exp( 2 (Z_0(m) + Z_\theta(m))\Big]  &\le  
\prod_{M+1\le m\le \log K_r} \frac{\sqrt{1+|\rho_m(\theta)|}}{\sqrt{1-|\rho_m(\theta)|}} \nonumber \\
&\times \E \Big[ \ind_{\mathcal E} \prod_{M+1\le m\le \log K_r} \exp( 2 (\widetilde{Z}_0(m) + \widetilde{Z}_\theta(m))\Big] \nonumber \\ 
&\ll  \E \Big[ \ind_{\mathcal E} \prod_{M+1\le m\le \log K_r} \exp( 2 (\widetilde{Z}_0(m) + \widetilde{Z}_\theta(m))\Big].
\end{align}
Here we estimated $\prod_{M+1\le m\le \log K_r} \sqrt{(1+|\rho_m(\theta)|}/\sqrt{(1-|\rho_m(\theta)|)}$ as $\ll 1$ using \eqref{12.6}, \eqref{12.7}, and our assumption that $e^M |\theta| \geq 10^{3}$.
Let us also clarify that the event ${\mathcal E}$ denotes (on the left side of \eqref{12.8}) the inequalities given in \eqref{11.7}, and on the right side of \eqref{12.8} these inequalities amount to, for all $M+1 \le m\le \log K_r$ 
\begin{equation} 
\label{12.9} 
\sum_{M+1 \le \ell \le m} {\widetilde Z_0}(\ell), \sum_{M+1\le \ell \le m} \widetilde{Z}_\theta(\ell) \le B +\sum_{e^M \le k < e^m} \frac{r^{2k}}{k}= B+\sum_{M\le \ell \le m} 2\sigma_\ell^2. 
\end{equation}

Since ${\widetilde Z}_0(m)$ and $\widetilde{Z}_\theta(m)$ are independent,  the right side of \eqref{12.8} equals 
$$ 
\Big( \E \Big[ \ind_{\mathcal E} \prod_{M+1\le m\le \log K_r} \exp(2 \widetilde{Z}_0(m)) \Big] \Big)^2, 
$$ 
where now by ${\mathcal E}$ we understand the constraints in \eqref{12.9} holding just for $\widetilde{Z}_0(m)$.  If we put $Y_m= \widetilde{Z}_0(m) - 2\sigma_m^2 (1+|\rho_m(\theta)|)$ then (completing the square as in Section 7) we obtain 
\begin{align*}
 \E \Big[ \ind_{\mathcal E}&\prod_{M+1\le m\le \log K_r} \exp(2 \widetilde{Z}_0(m)) \Big] = 
 \exp\Big( \sum_{M+1\le m\le \log K_r} 2\sigma_m^2 (1+|\rho_m(\theta)|) \Big) \\ 
&\times  \intmult_{\mathcal{E}'}  \exp\Big( - \sum_{M+1\le m \le \log K_r} \frac{Y_m^2}{2\sigma_m^2 (1+|\rho_m(\theta)|)} \Big) \prod_{M+1\le m\le \log K_r} 
  \frac{dY_m}{\sqrt{2\pi \sigma_m^2 (1+|\rho_m(\theta)|)}}, 
  \end{align*} 
 where ${\mathcal E}'$ now denotes the constraint  
 $$ 
 \sum_{M+1 \le \ell \le m}  Y_m \le B + \sum_{M+1\le \ell \le m} \big( 2\sigma_\ell^2 - 2 \sigma_\ell^2 (1+ |\rho_\ell(\theta)|) \big) = B- \sum_{M+1\le \ell \le m} 2\sigma_\ell^2 |\rho_\ell(\theta)|.  
 $$ 
We conclude that  
 \begin{align*} 
 \E \Big[ \ind_{\mathcal E}&\prod_{M+1\le m\le \log K_r} \exp(2 \widetilde{Z}_0(m)) \Big] \le  \exp\Big( \sum_{M+1\le m\le \log K_r} 2\sigma_m^2 (1+|\rho_m(\theta)|) \Big)\\ 
 &\times  \P \Big[ \sum_{M+1 \le \ell \le m}  Y_m \le B \text{  for  all  } M+1 \le m\le \log K_r \Big] \\ 
 &\ll  \exp\Big( \sum_{M+1\le m\le \log K_r} 2\sigma_m^2 (1+|\rho_m(\theta)|) \Big) \frac{ 1+ \max(0, B)}{\sqrt{1+ \log (K_r/e^M)}}, 
 \end{align*} 
  upon appealing to \cref{lem:Harper}.  Finally recalling \eqref{12.4} and \eqref{12.7} we have 
  $$ 
   \sum_{M+1\le m\le \log K_r} 2\sigma_m^2 (1+|\rho_m(\theta)|) = \sum_{e^M \le k < K_r} \frac{r^{2k}}{k} + O(1) = 
   \log K_r - M  + O(1).  
   $$ 
 This establishes \cref{prop11.1}, and hence the lower bound in \cref{thm:main}.    \hfill \qed

\bibliographystyle{plain}
\bibliography{bibtex_library}{}
\end{document}